\documentclass[a4paper,12pt, reqno]{amsart}
\usepackage{a4wide}
\usepackage{amsthm}
\usepackage{amssymb}
\usepackage{palatino}
\usepackage{lipsum}
\usepackage{longtable}
\usepackage{enumitem}
\usepackage{hyperref}

\usepackage{rotating}
\newtheorem{thm}{Theorem}[section]

\newtheorem{lem}[thm]{Lemma}

\theoremstyle{definition}
\newtheorem{defn}[thm]{Definition}

\newtheorem{exmp}[thm]{Example}

\theoremstyle{remark}
\newtheorem{rem}[thm]{Remark}

\numberwithin{equation}{section}

\newcommand{\legendre}[2]{\genfrac{(}{)}{}{}{#1}{#2}}
\newcommand{\Mod}[1]{\ (\mathrm{mod}\ #1)}

\begin{document}
	%%% \topmatter
	\title[A note on balancing sequences and application to cryptography]{A note on balancing sequences and application to cryptography}
	\author[K. Anitha, I. Mumtaj Fathima and A R Vijayalakshmi]{K. Anitha$^{(1)}$, I. Mumtaj Fathima$^{(2)}$ and A R Vijayalakshmi$^{(3)}$}
	\address{$^{(1)}$Department of Mathematics, SRM IST Ramapuram, Chennai 600089, India}
	\address{$^{(2)}$Research Scholar, Department of Mathematics, Sri Venkateswara College of Engineering \\ Affiliated to Anna University, Sriperumbudur, Chennai 602117, India}
	\address{$^{(3)}$Department of Mathematics, Sri Venkateswara College of Engineering, Sriperumbudur, Chennai 602117, India}
	\email{$^{(1)}$subramanianitha@yahoo.com}
	\email{$^{(2)}$tbm.fathima@gmail.com}
	\email{$^{(3)}$avijaya@svce.ac.in}

	%\email[$^{(1)}K Srinivas]{srini@imsc.res.in}
	%\email[$^{(2)}M Subramani]{msubramani@hri.res.in}
	%\email[$^{(3)}Usha K Sangale]{ushas073@gmail.com}

	\begin{abstract}
	In this paper, we prove the lower bound for the number of balancing non-Wieferich primes in arithmetic progressions. More precisely, for any given integer $r\geq2$ there are $\gg\log x$ balancing non-Wieferich primes $p\leq x$ such that $p\equiv\pm1 \Mod{r}$, under the assumption of the $abc$ conjecture for the number field $\mathbb{Q}(\sqrt{2})$. Further, we discuss some applications of balancing sequences in cryptography.
	\end{abstract}
	
	\maketitle
	\section{Introduction}
	
In $1909$, Arthur Wieferich \cite{wief} established the connection between first case of Fermat's last theorem and Wieferich primes. More precisely, if the first case of Fermat's last theorem fails for an odd prime $p$, then $p$ is a Wieferich prime for base $2$. The Wieferich primes are defined below: 

 Let $b \geq2$ be an integer. An odd prime $p$ is called a \textit{Wieferich prime for base $b$}
if
$$
b^{p-1}\equiv 1 \Mod{ p^2}.
$$
Otherwise it is called a \textit{non-Wieferich prime for base $b$}. 
%If we replace $b$ by $2$, it is simply called a \textit{Wieferich prime}. \\
%i.e.,
%\begin{equation}\label{e1}
%	2^{p-1}\equiv 1 \Mod{ p^2}.
%\end{equation}
%If the equation \eqref{e1} is not satisfies for a prime $p$ is called a \textit{non-Wieferich prime}. That is $p$ satisfies the congruence,
%$$
%2^{p-1}\not\equiv 1 \Mod{ p^2}.
%$$
%In 1909, A. Wieferich found a connection between non-Wieferich primes and Fermat last Theorem. In particular he proved that, if $p$ is a non-Wieferich prime then there is no integer solution to the first case of Fermat last theorem $x^p+y^p=z^p$ and $x, \, y, \, z$ are not divisible by $p$ (see \cite{wief}).
%It is well-known that Wieferich prime for base $2$ and the first case of Fermat's last theorem are strongly intertwined \cite{wief}.
As of today, the only known Wieferich primes for base $2$ are $1093$ and $3511$.
\medskip

A search for the Wieferich prime is one of the long-standing problems in number theory. It is still unknown whether there are finitely or infinitely many Wieferich primes that exist for any base $b\geq2$. However, Silverman \cite{sman} established the conditional results on non-Wieferich primes. Assuming the $abc$ conjecture \cite{gyo}, he proved the infinitude of non-Wieferich primes for any base $b$.

\medskip
\noindent
For any fixed $b\in \mathbb{Q}^*$, where  $\mathbb{Q}^*=\mathbb{Q}\backslash\{0\}$ and $b\neq\pm1$, if the $abc$ conjecture is true, then
$$
	\big|\{primes \, p\leq x: b^{p-1}
	\not\equiv 1 \Mod {p^2} \}\big|\gg_b \log x.
$$
In $2013$, Graves and Ram Murty \cite{graves} improved Silverman's result to certain arithmetic progressions. They showed that, if $b$ and $r$ are positive integers and assume the $abc$ conjecture, then
$$
	\big|\{primes \, p\leq x: p\equiv 1 \Mod r,\\b^{p-1}
	\not\equiv 1 \Mod {p^2}\}\big|\gg \frac{\log x}{\log\log x}.
$$
Then, there has been further enhancement made by Chen and Ding \cite{chending}.
$$
	\big|\{primes \, p\leq x: p\equiv 1 \Mod r,\\b^{p-1}
	\not\equiv 1 \, \Mod {p^2}\}\big|\gg \frac{\log x\,(\log\log\log x)^M}{\log\log x},
$$
where $M$ is any fixed positive integer.
Recently, Ding \cite{ding}  further strengthened the lower bound as
$$
	\big|\{primes \, p\leq x: p\equiv 1 \Mod r,\\b^{p-1}
	\not\equiv 1 \Mod {p^2}\}\big|\gg\log x.
$$

In this paper, we prove the similar lower bound for non-Wieferich primes in balancing numbers $\{B_n\}$ (defined in Section \ref{balancing}), assuming the abc conjecture for the number field $\mathbb{Q}(\sqrt{2})$.

\medskip
In $1999$, Behera and Panda \cite{behera} first proposed the concept of balancing numbers and studied their properties. Then, Rout \cite{rout} defined the balancing Wieferich primes as follows:

	An odd prime $p$ is called a \textit{balancing Wieferich prime} if it satisfies the congruence
$$
\displaystyle B_{p-\legendre{8}{p}}\equiv 0 \Mod{p^2},
$$
where $\legendre{8}{p}$ denotes the Legendre symbol. Otherwise, it is called a \textit{balancing non-Wieferich prime}.
%It is a natural question to ask, are there finite or infinitely many balancing non-Wieferich primes?
Recently, Wang and Ding \cite{wang} proved that there are  $\gg\log x$ balancing non-Wieferich primes, assuming the $abc$ conjecture for the number field $\mathbb{Q}(\sqrt{2})$. Earlier, Rout \cite{rout} and Dutta et al. \cite{dutta} proved some lower bounds for the number of balancing non-Wieferich primes $p$ such that $p\equiv1\Mod{r}$, where $r\geq2$ be a fixed integer. %improved the lower bound to $\log x/\log \log x$ and $(\log x/\log \log x) (\log \log\log x)^M$, where $M$ is any fixed positive integer respectively.
However, Wang and Ding \cite{wang} remarked that their results had some gaps. To the best of our knowledge, the main theorem in this paper is the first result in this direction which addresses the problem of balancing non-Wieferich primes in arithmetic progressions.

\medskip
\noindent
More precisely, we prove the following main theorem: 
\begin{thm}\label{theorem2}
	Let $r\geq2$ be any fixed integer and let $n>1$ be any integer. Assume that the $abc$ conjecture for the number field $\mathbb{Q}(\sqrt{2})$ is true. Then
	\begin{align*}
		% \nonumber % Remove numbering (before each equation)
		\big|\big\{primes \,  p\leq x: p &\equiv \pm1 \Mod{r}, \, 
		B_{p-\legendre{8}{p}}\not\equiv 0 \Mod{p^2}\big\}\big|
		\gg_{\alpha, r} \log x.
	\end{align*}
\end{thm}
Further, as an application in cryptography, the various studies related to public key encryption-decryption schemes based on the recurrence sequences have been noted. In $2014$, Ray et al. \cite{ray} developed a scheme using finite state machines, recurrence relation of balancing sequences, and balancing matrices. Viswanath and Ranjith kumar \cite{viswanath} proposed the concept of public-key cryptography using Hill cipher techniques and developed the cryptosystem using rectangular matrices. Then, further enchancement made by Sundarayya and Prasad \cite{sundarayya} using Affine-Hill cipher techniques. Recently, Prasath and Mahato \cite{prasad} proposed a public-key cryptosystem using Affine-Hill cipher with generalized Fibonacci matrix and discussed its strength.

In Section \ref{crypto}, we propose a public-key cryptosystem using \textit{Affine-Hill cipher} with a generalized balancing matrix with large power $k$, i.e., $Q_{B_{s}}^{k}$ as a key. We exchange the key matrix $K=Q_{B_{s}}^{k}$ of order $s\times s$ for encryption-decryption scheme with the help of balancing sequences under prime modulo. Instead of exchanging a key matrix, in this scheme we simply need to trade a pair of numbers $(s, k)$, which results in a wide key-space and lower time and space complexity.
\section{Preliminaries}\label{sec}
\subsection{Balancing numbers}\label{balancing}
The sequence of balancing numbers $\{B_n\}$ is defined by the recurrence relation 
\begin{equation}\label{recurrencebalance}
B_{n+1}=6B_n-B_{n-1}	
\end{equation}
 for $n\geq1$ with initial conditions $B_0=0$ and $B_1=1$. 
\begin{defn}\cite{behera}
	A positive integer $n$ is called a \textit{balancing number}
if
$$
		1+2+\ldots+(n-1)=(n+1)+(n+2)+\ldots+(n+l),
$$
where $l\in \mathbb{Z^+}$ is called the \textit{balancer} corresponding to the balancing number $n$.
\end{defn}
In other words, $n\in \mathbb{Z^+}$ is a balancing number if and only if $n^2$ is a triangular number. i.e., $8n^2+1$ is a perfect square.\\
 The \textit{Binet formula} for balancing number is
\begin{equation*}
	B_n=\frac{\alpha^n-\beta^n}{\alpha-\beta},
\end{equation*}
where $\displaystyle\alpha=3+2\sqrt{2}$ and $\displaystyle\beta=3-2\sqrt{2}$.\\
 Throughout this paper, we take $\displaystyle\alpha=3+2\sqrt{2}$ and $\displaystyle\beta=3-2\sqrt{2}$.
%\begin{defn}(G. K. Panda and S. S. Rout \cite{paro})
%	A number $l\in \mathbb{N}$ is called a \textit{period} of the balancing sequence modulo any natural number $m$ if $B_l\equiv0 \Mod  m$ and $B_{l+1}\equiv1 \Mod m$ and if any $n\in \mathbb{N}$ satisfies these congruences, that is $B_n\equiv0 \Mod  m$ and $B_{n+1}\equiv1 \Mod m$. Then $l$ divides $n$.
%\end{defn}
%The periodicity of balancing numbers mostly helps in determining the divisibility conditions of balancing numbers. Moreover, G. K. Panda and S. S. Rout \cite{paro} conjectured that there are only three prime numbers $13,\, 31,$ and $1546463$ such that periods of balancing sequence modulo those prime numbers are equal to the periods modulo their squares.
Further we note that, for any prime $p>2, \, B_{p-\legendre{8}{p}}\equiv 0 \Mod p$ \cite{paro}. 

\subsection{The $abc$ conjecture}
The $abc$ conjecture was formulated by Oesterl\'{e}\cite{oesterle} and Masser \cite{masser}. We state the $abc$ conjecture \cite{gyo} below: \\For any given real number $\varepsilon>0$, there is a constant $C_{\varepsilon}$ which depends only on $\varepsilon$ such that for every triple of positive integers $a, \, b, \,c$ satisfying $a+b=c$ with $\gcd(a, b)=1$, we have
$$
	c<C_{\varepsilon}(rad(abc))^{1+\varepsilon},
$$
where $rad(abc)=\textstyle\prod\limits_{p|abc}p$.\\
We now recall the definition of Vinogradov symbol.
\begin{defn}\cite{vojta}
	Let $f$ and $g$ are two non-negative functions. If $f<cg$ for some positive constant $c$, then we write $f\ll g$ or $g\gg f$. It is also called \textit{Vinogradov symbol}.	
\end{defn}

\subsubsection{The $abc$ conjecture for number fields(\cite{vojta}, \cite{gyo})}
Let $K$ be an algebraic number field and $K^{*}=K\backslash\{0\}$. Let $V_K$ be the set of primes on $K$, that is any $\upsilon\in V_K$ is an equivalence class of non-trivial norms on $K$ (finite or infinite). Let
$\|x\|_\upsilon:=N_{K/\mathbb{Q}}(\mathfrak{p})^{-\upsilon_\mathfrak{p}(x)}$, if $\upsilon$ (finite) is defined by a prime ideal $\mathfrak{p}$ of the ring of integers $\mathcal{O}_K$ in $K$ and $\upsilon_\mathfrak{p}$ is the corresponding valuation, where $N_{K/\mathbb{Q}}$ is the absolute value norm. For $\upsilon$ is infinite and let $\|x\|_\upsilon:=|\rho(x)|^e$ for all non-conjugate embeddings $\rho:K\rightarrow\mathbb{C}$ with $e=1$ if $\rho$ is real and $e=2$ if $\rho$ is complex.
\medskip

The \textit{height} of any triple $(a, \, b, \, c) \in K^*$ is 
$$
	H_{K}(a, \,b, \, c):=\displaystyle\prod\limits_{\upsilon\in V_K}\max(\|a\|_{\upsilon}, \, \|b\|_{\upsilon}, \, \|c\|_{\upsilon}).
$$
The \textit{radical} of the triple $(a, \, b, \, c) \in K^*$ is 
$$
	rad_K(a, \, b, \, c):=\displaystyle\prod\limits_{\mathfrak{p}\in I_{K}(a, \, b, \,c)}N_{K/\mathbb{Q}}(\mathfrak{p})^{\upsilon_\mathfrak{p}(p)},
$$
where $p$ is a rational prime with $p\mathbb{Z}=\mathfrak{p}\cap\mathbb{Z}$ and $I_K(a, \, b, \,c)$ is the set of all prime ideals $\mathfrak{p}$ of $\mathcal{O}_K$ for which $\|a\|_\upsilon, \, \|b\|_\upsilon, \, \|c\|_\upsilon$ are not equal.
\medskip

The $abc$ conjecture for algebraic number field $K$ states that for any $\varepsilon>0$ there exists a positive constant $C_{K,\varepsilon}$ such that
$$
	H_{K}(a, \,b, \, c)\leq C_{K, \varepsilon}(rad_K(a,\, b, \, c))^{1+\varepsilon},
$$
for all $a, \, b, \, c\in K^*$ satisfying $a+b+c=0$.
\subsection{Cyclotomic polynomial}
We now recall the cyclotomic polynomial and some of its properties.
\begin{defn}\cite{ram}
	 For any integer $m\geq1$, the $m^{th}$ \textit{cyclotomic polynomial} is
$$
		\Phi_m(X)=\prod\limits_{\substack{i=1\\ \gcd(i,m)=1}}^{m}(X-\zeta_m^i),
$$
	where $\zeta_m$ is the primitive $m^{th}$ root of unity.
\end{defn}
It follows that the recursion formula for cyclotomic polynomial is
\begin{equation}\label{se}
	X^m-1=\prod\limits_{\substack{d|m }}\Phi_d(X).
\end{equation}
%The following result found in \cite{thang} which connects the cyclotomic polynomial and Euler totient function $\phi$.
%\begin{lem}\label{vatwani}
	%For all integers $m\geq2$ and $a\geq2$, then
	% \begin{equation*}
		%\Phi_m(a)\geq\frac{1}{2}a^{\phi(m)}.
		%\end{equation*}
	%\end{lem}
The following lemma characterizes the prime divisors of $\Phi_m(\alpha,\beta)$, \\ where $$
\Phi_m(\alpha,\beta)=\prod\limits_{\substack{i=1\\ \gcd(i,m)=1}}^{m} (\alpha-\zeta_m^i\beta).
$$
\begin{lem}(Stewart \cite[Lemma 2]{stewart})\label{rm}
Let $(\alpha+\beta)^2$ and $\alpha\beta$ be coprime non-zero integers with $\alpha/\beta$ not a root of unity. If $m>4$ and $m\neq6,  12$ then $P(m/\gcd(3,m))$ divides $\Phi_m(\alpha,\beta)$ to at most the first power.  All other prime factors of $\Phi_m(\alpha, \beta)$ are congruent to $\pm1 \Mod{m}$. Further, if $m>e^{452}4^{67}$ then $\Phi_m(\alpha, \beta)$ has at least one prime factor congruent to $\pm1 \Mod{m}$.	
\end{lem}
Here, $P(k)$ denotes the greatest prime factor of $k$ with the convention that $P(0)=P(\pm1)=1$. We note that, Yu. Bilu et al. \cite{bilu} reduced the above lower bound $e^{452}4^{67}$ to $30$. In the Lemma \ref{rm}, Stewart \cite{stewart} considered the cyclotomic polynomial 
\begin{equation}\label{stewart cyclo}
		\alpha^m-\beta^m=\prod\limits_{\substack{d|m }}\Phi_d(\alpha, \beta).
\end{equation}
Since $\beta$ is a unit in $\mathbb{Q}(\sqrt{2})$, we notice that the prime divisors of $\Phi_m(\alpha, \beta)$ and the prime divisors of $\Phi_m(\alpha/\beta)$ are the same. Thus by using above Lemma \ref{rm}, the prime divisors of $\Phi_m(\alpha/\beta)$ are congruent to $\pm1 \Mod{m}$. 
%If we replace integer $a$ with real number $\alpha$ in Lemma \ref{vatwani}, a similar type of inequality holds for real $\alpha$, which is available in \cite{rout}. It is stated as follows.
\begin{lem}(Rout \cite[Lemma 2.10]{rout})\label{rg}
	For any real number $b$ with $|b|>1$, there exists $C>0$ such that
	\begin{equation*}
		|\Phi_m(b)|\geq C|b|^{\phi(m)},
	\end{equation*}
	where $\phi(m)$ is Euler's totient function.
\end{lem}
\subsection{Some lemmas}
\noindent We state some of the important lemmas from \cite{rout}, \cite{wang} and \cite{ding}.
\begin{lem}(Rout \cite[Lemma 2.12]{rout})\label{bn}
	Suppose that $B_n$ factored into $X_nY_n$, where $X_n$ and $Y_n$ are square-free and powerful part of $B_n$ respectively. If $p|X_n,$ then
$$
		B_{p-\big(\frac{8}{p}\big)}\not\equiv 0 \Mod{p^2}.
$$
\end{lem}
\begin{lem}(Rout \cite[Lemma 2.9]{rout})\label{an}
	For any $n\geq2$, the $n^{th}$ balancing number satisfies the following inequality.
$$
		\alpha^{n-1}<B_n<\alpha^{n}.
$$
\end{lem}
\begin{lem}(Rout \cite{rout})\label{ec}
	If the $abc$ conjecture for the number field $\mathbb{Q}(\sqrt{2})$ is true, then $ Y_{nr}\ll_\varepsilon B^{2\varepsilon}_{nr}$.
\end{lem}
This result is part of the proof of \cite[Theorem 3.1]{rout}.
%\begin{lem}(\cite[Lemma 3.3]{dutta})\label{l9}
%Suppose that the $abc$ conjecture for the number field $\mathbb{Q}(\sqrt{2})$ is true. There exists an integer $n_0$ depending only on $\alpha, \, r, \, M$ such that if $n\in\tau_{M}$ with $n\geq n_0,$ then $X^\prime_{nr}\gg nr$.
%\end{lem}
%\begin{proof}
%Let $\epsilon=\frac{\delta_{M}\phi(r)}{2r}$. Now from equation \eqref{et} we have,
%\begin{equation}\label{fe}
% X^\prime_{nr}Y^\prime_{nr}\gg \alpha^{2\phi(n)\phi(r)}
%\end{equation}
% and from equation \eqref{ec} we have
%\begin{equation*}
% Y_{nr}\ll_{\epsilon} B^{2\epsilon}_{nr}
%\end{equation*}
%By using Lemma \eqref{an} and equation \eqref{ec} we obtain,
%\begin{equation}\label{fs}
% Y^\prime_{nr}\ll Y_{nr}\ll_{\epsilon} B^{2\epsilon}_{nr}<(\alpha^{nr})^{2\epsilon}
%\end{equation}
% On substituting \eqref{fs} into \eqref{fe} we get,
%\begin{equation*}
% X^\prime_{nr}\gg \alpha^{2(\phi(nr)-\epsilon nr)}.
%\end{equation*}
%Now we write
% \begin{equation*}
	%  \phi(nr)-\epsilon nr\geq\phi(n)\phi(r)-\epsilon nr.
	%\end{equation*}
	%On taking $n\in\tau_M$, we get $\phi(n)\geq n\delta_M$. After some simple calculations, we conclude that
	% \begin{eqnarray*}
		% \nonumber % Remove numbering (before each equation)
		%\phi(nr)-\epsilon nr&=& n\delta_{M}\phi(r)-\epsilon nr \\
		% &=& \epsilon nr.
		%\end{eqnarray*}
		% Hence,
		%\begin{equation*}
		% X^\prime_{nr}\gg \alpha^{2\epsilon nr}>B^{2\epsilon }_{nr}> nr.
		% \end{equation*}

	%\end{proof}
	\begin{lem}(Wang and Ding \cite[Lemma 2.4]{wang})\label{l8}
		If $m<n$, then $\gcd(X^\prime_{m}, X^\prime_{n})=1$ or a power of $\sqrt{2}$.
	\end{lem}
	%\begin{proof}
	%We assume that $gcd(X^\prime_{nr}, \, X^\prime_{ir})>1$, for $i<n$, that is there exists a prime $p$ such that $p|X^\prime_{nr}$ and $p|X^\prime_{ir}$. By the definitions of $X^\prime_{nr}$ and $X^\prime_{ir}$, we say that $p|B_{nr}, \, p|B_{ir}$. So that $p|gcd(B_{nr}, \, B_{ir})$. By the divisibility conditions of $B_{n}$ we write,
	%$gcd (B_{nr}, B_{ir})=\displaystyle B_{gcd(nr, \, ir)}$ (see \cite{panda}). Thus $p|\displaystyle B_{gcd(nr, \, ir)}$. We write,
	% \begin{equation*}
		% \nonumber % Remove numbering (before each equation)
		%B_{nr} =\frac{B_{nr}}{\displaystyle B_{gcd(nr, \, ir)}}\displaystyle B_{gcd(nr, \, ir).}
		% \end{equation*}
	% Since $\Phi_{nr}(\alpha/\beta)| \Phi_1(\alpha/\beta)B_{nr}$, we obtain
	%\begin{equation*}
	% \Phi_{nr}(\alpha/\beta) | \Phi_1(\alpha/\beta)\displaystyle\frac{B_{nr}}{B_{gcd(nr, \, ir)}}B_{gcd(nr, \, ir)}.
	% \end{equation*}
% Thus,
%\begin{equation*}
%  \Phi_{nr}(\alpha/\beta)|\displaystyle\frac{B_{nr}}{B_{gcd(nr, \, ir)}}
%\end{equation*}
%as $gcd(\Phi_{nr}(\alpha/\beta), \, \Phi_1(\alpha/\beta))=1$ and $gcd(\Phi_{nr}(\alpha/\beta), \, \displaystyle B_{gcd(nr,ir)})=1$. Since $p|B_{gcd(nr, \, ir)}$, we arrive $p^2|B_{nr}$, which contradicts our assumption on $p|X^\prime_{nr}$. Hence $gcd(X^\prime_{nr}, \, X^\prime_{ir})=1$.
%\end{proof}

% The proof of the following lemma available in \cite{ding}.
\begin{lem}(Ding \cite[Lemma 2.5]{ding})\label{lim}
	For any given positive integers $r$ and $n$, we have
	$$
	\displaystyle \sum_{n\leq x}\frac{\phi(nr)}{nr}=c(r)x+O(\log x),
	$$
	where $c(r)=\displaystyle\prod_{p}\bigg(1-\frac{\gcd(p, r)}{p^2}\bigg)>0$ and the implied constant depends on $r$.
\end{lem}
\subsection{Hill-Cipher}
The idea of the Hill cipher is to use matrices to encrypt blocks of characters by replacing a block of a small number of letters with another block of the same size. The Hill encryption scheme replaces $n$ consecutive plaintext letters with $n$ ciphertext letters. The matrix representation for plaintext $A$, key matrix $K$, and cipher text $C$ are given as
\begin{align*}
	A =
	\begin{pmatrix}
		A_{1} & A_2 & \cdots & A_m	
	\end{pmatrix},
\end{align*}

\begin{align*}
	K =
	\begin{pmatrix}
		K_{1,1} & K_{1,2} & \cdots & K_{1,n}\\
		K_{2,1} & K_{2,2} & \cdots & K_{2,n}\\
		\vdots & \vdots & \vdots & \vdots\\
		K_{n,1}& K_{n,2} & \cdots & K_{n,n}
	\end{pmatrix},
\end{align*}
\begin{align*}
	C &=
	\begin{pmatrix}
		C_{1} & C_2 & \cdots & C_m	
	\end{pmatrix},
\end{align*}
where $A_i$ and $C_i$ are block matrices of size $1\times n$. Thus the Hill cipher is described as follows:\\
For encryption 
$$
\text{Enc}(A): C_i\equiv A_iK \Mod{p}.
$$
For decryption
$$
\text{Dec}(A): A_i\equiv K^{-1}C_i \Mod{p},
$$
where $p$ is a rational prime and $\gcd(det(K), p)=1$.
\subsection{Affine cipher}
An affine cipher of blocklength one is given by 
$$
\text{Enc}(x_i): y_i\equiv(ax_i+b) \Mod{26},
$$
where $a,b \in \mathbb{Z}_{26}$.
For decryption, we have to solve the function for $x_i$. So that, $a^{-1}\Mod{26}$ must exist, but $b$ can be any element of $Z_{26}$.
\subsection{Affine-Hill Cipher}
Affine-Hill Cipher is a polygraphic block cipher that extends the concept of Hill cipher by using the following encryption and decryption techniques.\\
For encryption 
$$
Enc(A): C_i\equiv(A_iK+B)  \Mod{p}.
$$
For decryption
$$
Dec(A): A_i\equiv (C_i-B)K^{-1} \Mod{p},
$$
where $A_i,C_i$, and $B$ are $1 \times n$ matrices and $K$ is $n\times n$ key matrix. Here, $p$ is prime greater than number of different characters used in plaintext. 
\subsection{Key exchange Technique (ElGamal algorithm)}
The ElGamal \cite{elgamal1}, \cite{elgamal2} cryptosystem can be constructed using any cyclic group in which the discrete-log problem is hard or believed to be hard. It will be broken if the discrete-log problem is solved. The cyclic group $\mathbb{F}_p^{\times}$ can be a multiplicative group of integers modulo $p$. In this technique, the global elements are the selected prime $p$ and chosen primitive root modulo $p$.
\subsection{Key Generation }
Choose a private key $G$ between $2$ and $p-2$ and select $g$ from primitive root modulo $p$, i.e., $g$ be the generator of $\mathbb{F}_p^{\times}$. Now we assign $E_1=g$. Computes $E_2=g^{G} \Mod{p}$. Suppose that \textit{Alice} and \textit{Bob} want to exchange key. Bob's public key is $(p, E_1, E_2)$ and his private key is $G$.
\subsubsection{Enciphering Stage}\label{secret}
\begin{enumerate}
	\item Alice can access Bob's public key  $(p, E_1, E_2)$.\\
	\item Alice chooses a random number $e$ such that $1<e<p-1$ and computes signature $k=E_1^{e}\Mod{p}$.\\
	\item Computes secret key $s=E_2^{e}\Mod{p}$.\\
	\item Calculates the cipher text $C=As \Mod{p}$\\
	\item Thus Alice can send encrypt message $(k, C)$ to Bob.
\end{enumerate}
\subsubsection{Deciphering Stage}
Bob uses his secret key to compute $s$.
\begin{align*}
	s &\equiv (E_2^{e}) \Mod{p}\\
	& \equiv (E_1^G)^{e} \Mod{p}\\
	& \equiv (E_1^{e})^{G} \Mod{p}\\
	& \equiv k^{G} \Mod{p}.
\end{align*}
Since both $k$ and $G$ are known to Bob, he can securely receive secret key $s$. Then by using Euclidean algorithm, $s^{-1}$ can be calculated. Thus, Bob will decrypt the cipher text $C$ and recover the plaintext $A=s^{-1}C$. 
\section{Main Results}
Let $r\geq 2$ be a given fixed integer and let $n>1$ be any integer. %The following theorem closely follows the Theorem 3.1 of \cite{rout}.

Let us take,
\begin{eqnarray*}
	% \nonumber % Remove numbering (before each equation)
	X^\prime_{nr} &=& \gcd(X_{nr}, \Phi_{nr}(\alpha/\beta)), \\
	Y^\prime_{nr} &=& \gcd(Y_{nr}, \Phi_{nr}(\alpha/\beta)).
\end{eqnarray*}
%The following result closely follows result in \cite{rout}.
The following theorem closely follows the proof of \cite[Theorem 3.1]{rout}. For the purpose of completeness we give the proof.

\begin{thm}\label{thm-4.1}
	Assume that the $abc$ conjecture for the number field $\mathbb{Q}(\sqrt{2})$ is true. Then for any $\varepsilon>0$,
	$
	X^\prime_{nr}
	\gg_{\varepsilon} B^{2(\phi(n)-\varepsilon)}_{\phi(r)}
	$.
\end{thm}
\begin{proof}
	By the recursion formula \eqref{se} we write,
	$$
		\Phi_{nr}(\alpha/\beta)=\displaystyle\frac{B_{nr}\Phi_1(\alpha/\beta)}{\beta^{nr-1}\displaystyle\prod_{d|nr}\Phi_{d}(\alpha/\beta)}.
	$$
	It follows that
	$$
		\Phi_{nr}(\alpha/\beta)|B_{nr}\alpha^{nr-1}.
	$$
	That is,
	$$
		\Phi_{nr}(\alpha/\beta)|X_{nr}Y_{nr}\alpha^{nr-1}.
	$$
%	As prime divisor of $\Phi_1(\alpha/\beta)$ is $\sqrt{2}$ and $\sqrt{2}\nmid\Phi_{nr}(\alpha/\beta)$, we say that $\gcd(\Phi_1(\alpha/\beta),\Phi_{nr}(\alpha/\beta))=1$.
Since $\alpha$ is a unit in $\mathbb{Q}(\sqrt{2})$, we have $\Phi_{nr}(\alpha/\beta)\nmid \alpha$. Thus, $	\Phi_{nr}(\alpha/\beta)|X_{nr}Y_{nr}$. As $\gcd(X_{nr},Y_{nr})=1$, we obtain either $\Phi_{nr}(\alpha/\beta)|X_{nr}$ or $\Phi_{nr}(\alpha/\beta)|Y_{nr}$. 

 We suppose that $\Phi_{nr}(\alpha/\beta)|X_{nr}$, it follows that $X^\prime_{nr}=\gcd(X_{nr},\, \Phi_{nr}(\alpha/\beta))=\Phi_{nr}(\alpha/\beta)$ and $Y^\prime_{nr}=\gcd(Y_{nr}, \, \Phi_{nr}(\alpha/\beta))=1$.
	Similar argument for $\Phi_{nr}(\alpha/\beta)|Y_{nr}$ implies that $X^\prime_{nr}=1$ and $Y^\prime_{nr}=\Phi_{nr}(\alpha/\beta)$. For any of these cases, we finally get
	\begin{equation}\label{ef}
		X^\prime_{nr}Y^\prime_{nr}=\Phi_{nr}(\alpha/\beta).
	\end{equation}
	By using Lemma \ref{rg} we write,
	\begin{eqnarray}
		% \nonumber % Remove numbering (before each equation)
		| X^\prime_{nr}Y^\prime_{nr}|&=&|\Phi_{nr}(\alpha/\beta)| \\
		&\geq& C|\alpha/\beta|^{\phi(nr)} \\
		&=& C|\alpha^2|^{\phi(nr)}.\label{es}
	\end{eqnarray}
	Since $\{B_{nr}\}$ is a sequence of positive integers, by using Lemma \ref{an} we get,
	\begin{eqnarray}
		% \nonumber % Remove numbering (before each equation)
		X^\prime_{nr}Y^\prime_{nr} &\geq&C\alpha^{2\phi(nr)}\label{this}\\
		&\geq& C(\alpha^{\phi(r)})^{2\phi(n)}  \\
		&>& C B^{2\phi(n)}_{\phi(r)}\label{et}.
	\end{eqnarray}
	Now by combining Lemma \ref{ec} with equation \eqref{et}, we obtain
	\begin{align*}
		% \nonumber % Remove numbering (before each equation)
		X^\prime_{nr}B^{2\varepsilon}_{nr} &\gg_{\varepsilon} X^{\prime}_{nr}Y_{nr}\geq X^\prime_{nr}Y^\prime_{nr}\gg_{\varepsilon} B^{2\phi(n)}_{\phi(r)}  \\
		X^\prime_{nr} &\gg_{\varepsilon} \frac{B^{2\phi(n)}_{\phi(r)}}{B^{2\varepsilon}_{nr}}.
	\end{align*}
	
	After simplification we write,
$$
		X^\prime_{nr}
		\gg_{\varepsilon} B^{2(\phi(n)-\varepsilon)}_{\phi(r)}.
$$

%We choose $\epsilon<1/2$ and we obtain,
%\begin{equation*}
	%  X^\prime_{nr}\gg B^{2\phi(n)-1}_{\phi(r)}.
	%\end{equation*}
%As $X^{\prime}_{nr}$ is a product of distinct primes in $\mathbb{Q}[\sqrt{2}]$ and when $n\rightarrow\infty$, then $\phi(n)\rightarrow \infty$. We get
%\begin{equation*}
	% \displaystyle lim_{n\rightarrow \infty}\#\{primes \,  p: p|X^{\prime}_{ir},\, i\leq n\}=\infty.
	%\end{equation*}
%Since prime $p$ divides $X^\prime_{nr}$ (implies $p|X_{nr}$) and from equation \eqref{ef} we can write $p$ divides $\Phi_{nr}(\alpha/\beta)$. Hence by Proposition \eqref{rm} and Lemma \eqref{bn} we conclude that there are infinitely many primes $p$ such that
%\begin{eqnarray*}
	% \nonumber % Remove numbering (before each equation)
	% B_{p-(\frac{8}{p})}&\not\equiv& 0\, (mod \, p^2) \\
	%p &\equiv&1 \, (mod\, r).
	% \end{eqnarray*}

This completes the proof of Theorem \ref{thm-4.1}.
\end{proof}
%Let $\tau _M$ be the set of all square-free integers with exactly $M+1$ prime factors and $\delta_M=\displaystyle\prod^{M+1}_{i=1}(1-\displaystyle\frac{1}{p_i})$, where $p_i$ be the $i$-th prime. %The following lemma closely follows Lemma 2.7 of \cite{chending}.

Let us take $T=\{n:\textstyle X^\prime_{nr}>nr\}$ and $T(x)=|T\cap[1, \, x]|$. The following lemma is an analogous result of \cite[Lemma 2.6]{ding} for the balancing sequences.
\begin{lem}\label{l11}
	We have $T(x)\gg x,$ where the implied constant depends only on $\alpha, \, r$.
\end{lem}
\begin{proof}
	Let $R=\left\{ n:\, \phi(nr)>\textstyle\frac{2c(r)}{3}nr\right\}$ and $R(x)=|R\cap[1, \, x]|$.\\
	By equation \eqref{this} we have,
	\begin{equation}\label{fe}
		X^\prime_{nr}Y^\prime_{nr}\gg \alpha^{2\phi(nr)}.
	\end{equation}
	% and from equation \eqref{ec} we have
	%\begin{equation*}
		%  Y_{nr}\ll_{\epsilon} B^{2\epsilon}_{nr}
		%\end{equation*}
	By using Lemmas \ref{an} and \ref{ec} we obtain,
	\begin{equation}\label{fs}
		Y^\prime_{nr}\leq Y_{nr}\ll_{\varepsilon} B^{2\varepsilon}_{nr}<(\alpha^{nr})^{2\varepsilon}.
	\end{equation}
	On substituting equation \eqref{fs} in \eqref{fe} we get,
	\begin{equation}\label{that}
		X^\prime_{nr}\gg_{\varepsilon} \alpha^{2(\phi(nr)-\varepsilon nr)}.
	\end{equation}
	Let $\varepsilon=\frac{c(r)}{3}$ in \eqref{that} and we get $ X^\prime_{nr}\gg_{r} \alpha^{2\big(\phi(nr)-\textstyle\frac{c(r)nr}{3}\big)}$. For any $n\in R$, we have $\phi(nr)>\frac{2c(r)nr}{3}$. Therefore,
	\begin{align*}
		% \nonumber % Remove numbering (before each equation)
		\displaystyle X^\prime_{nr} &\gg_{r} \displaystyle \alpha^{2\big(\phi(nr)-\frac{c(r)nr}{3}\big)}> \alpha^{\frac{2c(r)nr}{3}}>nr.
	\end{align*}
	Therefore there exists an integer $n_0$ depending only on $\alpha, \, r$ such that, if $n\geq n_0$ and $n\in R$, then $X^\prime_{nr}>nr$. Hence we obtain,
	\begin{align*}
		% \nonumber % Remove numbering (before each equation)
		T(x) &= \sum\limits_{\substack{n\leq x \\ X^\prime_{nr}>nr}}1\geq \sum\limits_{\substack{n\leq x \\ n\geq n_0 \\ n\in R }}1 
		= \sum\limits_{\substack{n\leq x \\ n\geq n_0 \\ \phi(nr)>2c(r)nr/3}}1. %\label{f1}.
	\end{align*}
	Since we note that,
	\begin{equation}\label{f2}
		% \nonumber % Remove numbering (before each equation)
		\sum\limits_{\substack{n\leq x \\ \phi(nr)\leq2c(r)nr/3}} \displaystyle\frac{\phi(nr)}{nr} \leq  \sum\limits_{\substack{n\leq x \\ \phi(nr)\leq2c(r)nr/3}} \displaystyle\frac{2c(r)}{3}
		\leq\displaystyle\frac{2c(r)}{3}x.
	\end{equation}
	Hence by Lemma \ref{lim} and equation \eqref{f2} we obtain,
	\begin{align*}
		% \nonumber % Remove numbering (before each equation)
		T(x) &\geq \sum\limits_{\substack{n\leq x \\ n\geq n_0 \\ \phi(nr)>2c(r)nr/3}}1 \\
		&\gg \sum\limits_{\substack{n\leq x \\ \phi(nr)>2c(r)nr/3}}1 \\
		&\geq \sum\limits_{\substack{n\leq x \\ \phi(nr)>2c(r)nr/3}}\frac{\phi(nr)}{nr} \\
		&= \sum\limits_{\substack{n\leq x }}\frac{\phi(nr)}{nr}-\sum\limits_{\substack{n\leq x \\ \phi(nr)\leq2c(r)nr/3}}\frac{\phi(nr)}{nr}\\
		&\geq c(r)x+O(\log x)-\textstyle\frac{2c(r)}{3}x\gg_{\alpha, r} x.
	\end{align*}
	This completes the proof of Lemma \ref{l11}.
\end{proof}
\subsection{Proof of Theorem \ref{theorem2}}
The main idea of this theorem is to count number of primes $p$ such that $p$ divides $X^\prime_{nr}\leq x$. For any $n\in T$, it follows that there exists an odd prime $p_{n}$ such that $p_n|X^\prime_{nr}$ and $p_n\nmid nr$. Since $X^\prime_{nr}|X_{nr}$ and $p_n|X^\prime_{nr}$, by using Lemma \ref{bn} we obtain
$$
	B_{p_{n}-\legendre{8}{p_n}}\not\equiv 0 \Mod {p_n^2}.
$$
We note that $p_n|X^\prime_{nr}, \, X^\prime_{nr}|\Phi_{nr}(\alpha/\beta)$ and $p_n\nmid nr$. Therefore, by using Lemma \ref{rm} we obtain $p_{n}\equiv\pm1 \Mod {nr}$. Thus for any $n\in T$, there is a prime $p_{n}$ satisfying,
\begin{align*}
	% \nonumber % Remove numbering (before each equation)
	B_{p_{n}-\legendre{8}{p_{n}}}&\not\equiv 0 \Mod { p_{n}^2}, \\
	p_{n} &\equiv \pm1 \Mod {nr}.
\end{align*}
By Lemma \ref{l8}, we get $p_n$ $(n\in T)$ are distinct primes. Thus we find that,
\begin{align*}
	% \nonumber % Remove numbering (before each equation)
	\big|\big\{primes \, p\leq x\, : \,  p \equiv \pm1 \Mod r,  \,
	B_{p-\legendre{8}{p}}\not\equiv 0 \Mod { p^2}\big\}\big|
	&\geq \big|\big\{n\, :\, n\in T, \, X^\prime_{nr}\leq x\big\}\big|.
\end{align*}
Since $X^\prime_{nr}\leq X_{nr}\leq B_{nr}<\alpha^{nr}$, we write
\begin{align*}
	% \nonumber % Remove numbering (before each equation)
	\big|\big\{n\, :\, n\in T, \, X^\prime_{nr}\leq x\big\}\big| &\geq \big|\{n\, :\, n\in T, \, \alpha^{nr}\leq x\}\big| \\
	&= \big|\big\{n\, :\, n\in T, \, n\leq \textstyle\frac{\log x}{r\log \alpha}\big \} \big|\\
	&= \textstyle T\big(\frac{\log x}{r \log \alpha}\big).
\end{align*}
Hence by Lemma \ref{l11} we conclude that,
\begin{align*}
	% \nonumber % Remove numbering (before each equation)
	\big|\big\{primes \, p\leq \, x\, : \,  p \equiv \pm1 \Mod{r}, \,  
	B_{p-\legendre{8}{p}}\not\equiv 0\Mod{p^2} \big\}\big|
	&\geq \textstyle T\big(\frac{\log x}{r \log \alpha}\big) \\
	&\gg_{\alpha, r} \log x.
	%\frac{\log x}{r \log \gamma}
\end{align*}
\section{Application to Cryptography}\label{crypto}

\subsection{Balancing Matrices}
Ray \cite{raymat} introduced the balancing $Q_B-$matrix of order $2$ whose entries are first three balancing numbers $0, 1,$ and $6$ as follows:
\begin{align*}
	Q_{B} &=
	\begin{pmatrix}
		B_2 & -B_1\\
		B_1 & B_0
	\end{pmatrix}
=
	\begin{pmatrix}
	6 & -1\\
	1 & 0
\end{pmatrix}
\end{align*}
Without loss of generality, we interchange the diagonal elements,\\
i.e., 
\begin{align*}
	Q_{B_2} &=
	\begin{pmatrix}
		B_0 & -B_1\\
		B_1 & B_2
	\end{pmatrix}
	=
	\begin{pmatrix}
		0 & -1\\
		1 & 6
	\end{pmatrix}
\end{align*}

The $n^{th}$ power of the balancing $Q_{B_{2}}$- matrix is
\begin{align*}
	Q_{B_2}^n &=
	\begin{pmatrix}
	-B_{n-1} & -B_n\\
		B_n & B_{n+1}
	\end{pmatrix}
\end{align*}
with $n>0$.\\
The \textit{Cassini Formula} \cite{ray} for balancing number is 
\begin{equation}\label{cassini}
	B_n^2-B_{n+1}B_{n-1}=1.
\end{equation}
Thus $det(Q_{B_2}^n)=1$ and $Q_{B_2}^n$ is a non-singular matrix for all $n$. Therefore inverse must exist.
\begin{align*}
(Q_{B_2}^n)^{-1} =Q_{B_2}^{-n}	 &=
	\begin{pmatrix}
		 B_{n+1} &B_n\\
		-B_n &-B_{n-1}
	\end{pmatrix}
\end{align*}
We now extend the balancing $Q_B-$ matrix of order 3,
\begin{align*}
	Q_{B_3}^n &=
	\begin{pmatrix}
		-B_{n-1} & -B_{n} & 0\\
		B_ {n}& B_{n+1} & 0\\
		\displaystyle\sum_{t=0}^{n-1}B_t & \displaystyle\sum_{t=1}^{n}B_t & 1
	\end{pmatrix}
\quad \text{with} \quad
Q_{B_3}=
	\begin{pmatrix}
		0 & -1 & 0\\
		1 & 6 & 0\\
		0 & 1 & 1
	\end{pmatrix},
\end{align*}
where $n>0$.
By continuing this process, 
\begin{align*}
	Q_{B_s} =
	\begin{pmatrix}
		0 & -1 &  0 & 0\cdots  0  & 0 &0\\
		1 &  6 &  0 & 0 \cdots 0 & 0& 0\\
		0 &  1 &  1 & 0 \cdots 0  & 0&0\\
		0 &  0 &  1 & 1 \cdots  0& 0&0\\
		\vdots & \vdots & \vdots & \ddots&\ddots&\vdots\\
		0 & 0 & 0 & \cdots 1&1&  0\\
		0 & 0 & 0 & \cdots 0& 1 & 1
	\end{pmatrix}.
\end{align*}
The $det(Q_{B_s})=1$ guaranteed the existence of inverse.\\ Thus 
\begin{align*}
	Q_{B_s}^{-1} =
	\begin{pmatrix}
		6 & 1 &  0 & 0\cdots  0  & 0 &0\\
	   -1 &  0 &  0 & 0 \cdots 0 & 0& 0\\
		1 &  0 &  1 & 0 \cdots 0  & 0&0\\
		-1 &  0 &  -1 & 1 \cdots  0& 0&0\\
		\vdots & \vdots &\vdots &\ddots & \vdots&\vdots \\
		(-1)^{s-2} & 0 & (-1)^{s-4} & \cdots & 1 & 0\\
		(-1)^{s-1} & 0 & (-1)^{s-3} & \cdots & -1& 1
	\end{pmatrix}.
\end{align*}
Now the $n^{th}$ power of $Q_{B_s}$ is
\begin{align*}
	Q_{B_s}^n =
	\begin{pmatrix}
	M_1 & \textbf{0}\\
	M_2 & M_3
	\end{pmatrix}, 
\end{align*}
here $Q_{B_s}^n$ is a block matrix and its blocks are 
\begin{align*}
		M_1=
	\begin{pmatrix}
 -B_{n-1} & -B_{n} & 0 \\
	B_ {n}& B_{n+1} & 0 \\
	\displaystyle\sum_{t=0}^{n-1} B_t & \displaystyle\sum_{t=1}^{n}B_t & 1 
	\end{pmatrix},
\end{align*}

\begin{align*}
	M_2=
	\begin{pmatrix}
		\scriptstyle B_0(n-1)+B_1(n-2)+\cdots+B_{n-2} & \scriptstyle B_1(n-1)+B_2(n-2)+\cdots+B_{n-1} & n \\
		\vspace{0.3cm} \\
		\scriptstyle B_0\frac{(n-1)(n-2)}{2} +B_1\frac{(n-2)(n-3)}{2}\\+\cdots +\scriptstyle B_{n-3}	 &  \scriptstyle B_1\frac{(n-1)(n-2)}{2}+B_2\frac{(n-2)(n-3)}{2}+\cdots+B_{n-2} & \frac{n(n-1)}{2} \\
		\vspace{0.3cm}\\
		\scriptstyle  B_0\displaystyle\scriptstyle \sum_{n>4}\scriptstyle\frac{(n-3)(n-4)}{2}+B_1\displaystyle\scriptstyle \sum_{n>4}\scriptstyle\frac{(n-3)(n-4)}{2}\\+\cdots+\scriptstyle B_{n-4} & \scriptstyle B_1	\displaystyle\scriptstyle\sum_{n>3}\scriptstyle\frac{(n-2)(n-3)}{2}+ \cdots+B_{n-3} & 1+\textstyle\displaystyle\scriptstyle\sum_{n>2}\scriptstyle\frac{(n-1)(n-2)}{2} \\
		\vspace{0.3cm}\\
		\scriptstyle B_0\displaystyle\scriptstyle\sum_{n>4}\scriptstyle\frac{(n-2)(n-3)(n-4)}{3!}\\+\scriptstyle B_1\displaystyle\scriptstyle\sum_{n>5}\frac{(n-3)(n-4)(n-5)}{3!}+\cdots +\scriptstyle B_{n-5} & \scriptstyle B_1\displaystyle\scriptstyle\sum_{n>4}\scriptstyle\frac{(n-2)(n-3)(n-4)}{3!}+\cdots+B_{n-4} & 1+	\displaystyle\scriptstyle\sum_{n>3}\scriptstyle\frac{(n-1)(n-2)(n-3)}{3!} \\
		\vdots & \vdots & \vdots\\
		0 & 0 & 0
		\end{pmatrix}
\end{align*}

\begin{align*}
	M_3=
	\begin{pmatrix}
			1 &  0  &0 & 0& \cdots &0\\
			\vspace{0.1cm}\\
		n \choose1 & 1& 0 & 0 &\cdots & 0\\
		\vspace{0.1cm}\\
		n\choose2 & n\choose1 & 1&0 &\cdots & 0\\
		\vspace{0.1cm}\\
		n\choose3 & n\choose2 & n\choose 1& 1 & \cdots & 0\\
		\vdots &  \vdots &  \vdots & \vdots &\ddots &  \vdots &\\
		n \choose n-1 & n\choose n-2 & n \choose n-3& n \choose n-4& \cdots & 1
	\end{pmatrix}
\end{align*}
and $\textbf{0}$ is a zero matrix of order $3\times (s-3)$.\\
Now, the inverse of $Q_{B_s}^{n}$ is 
\begin{align*}
	Q_{B_s}^{-n} =
	\begin{pmatrix}
		M_1^{-1} & \textbf{0}\\
		M_2^{-1} & M_3^{-1}
	\end{pmatrix}, 
\end{align*}
where 
\begin{align*}
	M_1^{-1}=
	\begin{pmatrix}
		B_{n+1} & B_{n} & 0 \\
		-B_ {n}& -B_{n-1} & 0 \\
	\displaystyle\sum_{t=1}^{n}B_t	 & \displaystyle\sum_{t=0}^{n-1} B_t & 1 
	\end{pmatrix},
\end{align*}
\begin{align*}
	M_2^{-1}=
	\begin{pmatrix}
		-\left(\scriptstyle nB_1+B_2(n-1)+\cdots+B_{n}\right) & -\left(\scriptstyle nB_0+B_1(n-1)+\cdots+2B_{n-2}+B_{n-1}\right) &-n \\
		\vspace{0.3cm}\\
		\scriptstyle B_1\frac{n(n+1)}{2} +B_2\frac{(n-1)n}{2}\\+\cdots +\scriptstyle 3B_{n-1}+B_n	 & \scriptstyle B_0\frac{n(n+1)}{2} +B_1\frac{n(n-1)}{2}+\cdots+B_{n-1} & \frac{n(n+1)}{2} \\
		\vspace{0.3cm}\\
		- (\scriptstyle B_1\displaystyle\scriptstyle \sum_{1}^{n}\scriptstyle\frac{(n+1)!}{(n-1)!2!}+B_2\displaystyle\scriptstyle\sum_{2}^{n}\scriptstyle\frac{n!}{(n-2)!2!}\\+\cdots+\scriptstyle 4B_{n-1}+B_n ) & -\left(\scriptstyle B_0
		\displaystyle\scriptstyle\sum_{1}^{n}\frac{(n+1)n}{2} +\scriptstyle B_1 \sum_{2}^{n}\scriptstyle\frac{n(n-1)}{2}+\cdots+B_{n-1}\right)  & -\left(\textstyle\displaystyle\scriptstyle\sum_{1}^{n}\scriptstyle\frac{(n+1)n}{2}\right) \\
		\vspace{0.3cm}\\
		\scriptstyle B_1\displaystyle\scriptstyle\sum_{1}^n\scriptstyle\frac{(n+2)!}{(n-1)!3!}\\+\scriptstyle B_2\displaystyle\scriptstyle\sum_{2}^n\frac{(n+1)!}{(n-2)!3!}+\cdots +\scriptstyle B_{n} & \left(\scriptstyle B_0\displaystyle\scriptstyle\sum_{1}^{n}\scriptstyle\frac{(n+2)(n+1)n}{3!}+\cdots+B_{n-1}\right) & \left(\displaystyle\scriptstyle\sum_{1}^{n}\scriptstyle\frac{(n+2)(n+1)n}{3!}\right)	 \\
		\vdots & \vdots & \vdots\\
		\scriptstyle(-1)^{s-1}(\scriptstyle B_1\scriptstyle\sum_{1}^{n}\scriptstyle\frac{(n+s-5)!}{(n-1)!(s-4)!}+\\\scriptstyle B_2\scriptstyle\sum_{2}^{n}\scriptstyle\frac{(n+s-6)!}{(n-2)!(s-4)!}+\\\scriptstyle B_3\scriptstyle\sum_{3}^{n}\scriptstyle\frac{(n+s-7)!}{(n-3)!(s-4)!}+\cdots+B_n) & \scriptstyle(-1)^{s-1}\left(\scriptstyle B_0\scriptstyle\sum_{1}^{n}\scriptstyle\frac{(n+s-5)!}{(n-1)!(s-4)!}+\scriptstyle B_1\scriptstyle\sum_{2}^{n}\scriptstyle\frac{(n+s-6)!}{(n-2)!(s-4)!}+\cdots+\scriptstyle B_{n-1}\right) &\scriptstyle (-1)^{s-1}\binom{n+s-4}{s-3}
	\end{pmatrix},
\end{align*}
\begin{align*}
	M_3^{-1}=
	\begin{pmatrix}
		1 &  0  &0 & 0& \cdots &0\\
		\vspace{0.1cm}\\
		-\binom{n}{n-1} & 1& 0 & 0 &\cdots & 0\\
		\vspace{0.1cm}\\
		\binom{n+1}{n-1} & -\binom{n}{n-1} & 1&0 &\cdots & 0\\
		\vspace{0.1cm}\\
	-\binom{n+2}{n-1} & \binom{n+1}{n-1} & -\binom{n}{n-1} & 1 & \cdots & 0\\
		\vdots &  \vdots &  \vdots & \ddots &\ddots &  \vdots &\\
	(-1)^{s-1}\binom{s+n-2}{n-1} & (-1)^{s-2}\binom{s+n-3}{n-1} & (-1)^{s-3}\binom{s+n-4}{n-1} & (-1)^{s-4}\binom{s+n-5}{n-1} & \cdots & 1
	\end{pmatrix}
\end{align*}
\subsection{Key Exchange Technique}
Assume that Alice and Bob want to exchange a key. Bob (the receiver) generates the components of the public key $E_1$ and $E_2$ using his private key (session key) $G$. Thus, the public key is $pk(p, E_1, E_2)$. With the help of this public key $pk(p, E_1, E_2)$, the secret key $s$ can be calculated (see \ref{secret}). Then the key matrix $K$ can be constructed using the secret key $s$.
\subsubsection{Algorithm}\label{key}
\textit{Enciphering Stage:}\\
\begin{enumerate}
	\item Alice chooses secret number $e$ such that $1<e<p-1$.\\
	\item Signature: $k\leftarrow E_1^e \Mod{p}$.\\
	\item Secret key: $s\leftarrow E_2^e\Mod{p}$.\\
	\item Key matrix: $K\leftarrow Q_{B_s}^k,$ where $Q_{B_s}^k$ is a generalized balancing matrix of order $s\times s$. \\
	\item Encryption: $Enc(A): \quad C_i\leftarrow(A_iK+B) \Mod{p}$.\\
	\item Exchange $(k, C)$ to Bob.
\end{enumerate}
\textit{Deciphering Stage:}\\
Bob after obtaining $(k, C)$, 
\begin{enumerate}
	\item Secret Key: $s\leftarrow k^G \Mod{p}$, where $G$ is Bob's secret key.\\
	\item Key Matrix: $K\leftarrow Q_{B_s}^k$.\\
	\item Decryption: $Dec(C): \quad A_i\leftarrow (C_i-B)K^{-1} \Mod{p}$.
	
\end{enumerate}
\subsection{Numerical Example}
Assuming Alice wants to send a message to Bob, she would first compute the key matrix $K$ with the help of the above algorithm \ref{key} and then encrypt the plaintext $A$ by using key matrix $K=Q_{B_s}^k$.
\begin{exmp}
Let $p=31$ and let Bob's private key $G=17$. Suppose that Alice wants to send a plaintext $A=$ \textbf{WELCOMEANNIE}.
Bob chooses (primitive root modulo $31$) $g$ as $3$. Assume that $g=E_1$ and computes $E_2\equiv E_1^{G} \Mod{p}\equiv3^{17} \Mod{31}=22$. Thus Bob's public key $pk(p, E_1, E_2)=(31, 3,22)$ and secret key $G=17$.\\
\textit{Enchipering Stage:}\\
The plaintext $A$ is WELCOMEANNIE and shifting vector is 
\begin{align*}
	B=
	\begin{pmatrix}
		37 & 17 & 19 & 13	
	\end{pmatrix},
\end{align*}
At first, Alice chooses $e=24$ and computes signature $k\equiv 3^{24} \Mod{31}\equiv2 \Mod{31}$.\\
The secret key $s\equiv E_2^{e}\Mod{p}\equiv 22^{24}\Mod{31}\equiv 4 \Mod{31}$. The key matrix $K$ can be constructed using aforementioned data with help of generalized balancing matrix $Q_{B_s}^{k}$.\\
i.e.,
\begin{align*}
K=Q_{B_4}^2 &=
	\begin{pmatrix}
		-1 & -6 & 0 & 0\\
		6 & 35 & 0 & 0\\
		1 & 7 & 1 & 0\\
		0 & 1 & 2 & 1
	\end{pmatrix}.
\end{align*}
The cipher text $C\leftarrow AK+B$. \\
Now, the plaintext $A$ can be divided into blocks.\\
i.e.,
\begin{align*}
	A_1=
	\begin{pmatrix}
		22 & 4 & 11 & 2
	\end{pmatrix};
\quad
A_2=
	\begin{pmatrix}
	14 & 12 & 4 & 0
\end{pmatrix};
\quad
A_3=
\begin{pmatrix}
	13 & 13 & 8 & 4
\end{pmatrix}.
\end{align*}
\begin{align*}
	C_1 &\equiv A_1K+B \Mod{31}\\
	&\equiv\left(\begin{pmatrix}
		22 & 4 & 11 & 2
	\end{pmatrix}
	\begin{pmatrix}
		-1 & -6 & 0 & 0\\
		6 & 35 & 0 & 0\\
		1 & 7 & 1 & 0\\
		0 & 1 & 2 & 1
	\end{pmatrix}+
 \begin{pmatrix}
 	37 & 17 & 19 & 13
 \end{pmatrix}\right)\Mod{31}\\
&\equiv\begin{pmatrix}
	19 & 11 & 3 & 15
\end{pmatrix}\sim
\begin{pmatrix}
	T & L & D & P
\end{pmatrix}
\end{align*}
\begin{align*}
	C_2 &\equiv A_2K+B \Mod{31}\\
	&\equiv\left(\begin{pmatrix}
		14 & 12 & 4 & 0
	\end{pmatrix}
	\begin{pmatrix}
		-1 & -6 & 0 & 0\\
		6 & 35 & 0 & 0\\
		1 & 7 & 1 & 0\\
		0 & 1 & 2 & 1
	\end{pmatrix}+
	\begin{pmatrix}
		37 & 17 & 19 & 13
	\end{pmatrix}\right)\Mod{31}\\
	&\equiv\begin{pmatrix}
		6 & 9 & 23 & 13
	\end{pmatrix}\sim
\begin{pmatrix}
	G& J& X & N
\end{pmatrix}
\end{align*}
\begin{align*}
	C_3 &\equiv A_3K+B\Mod{31}\\
	&\equiv\left(\begin{pmatrix}
		13 & 13 & 8 & 4
	\end{pmatrix}
	\begin{pmatrix}
		-1 & -6 & 0 & 0\\
		6 & 35 & 0 & 0\\
		1 & 7 & 1 & 0\\
		0 & 1 & 2 & 1
	\end{pmatrix}+
	\begin{pmatrix}
		37 & 17 & 19 & 13
	\end{pmatrix}\right)\Mod{31}\\
	&\equiv\begin{pmatrix}
		17 & 20 & 4 & 17
	\end{pmatrix}\sim
	\begin{pmatrix}
		R& U& E & R
	\end{pmatrix}
\end{align*}
Thus the cipher text $C=(C_1C_2C_3)=(TLDPGJXNRUER)$. Alice now send this cipher text to bob with her signature.\\
\textit{Deciphering Stage:}\\
Bob can calculate decryption key $K^{-1}$ with the the help of secret key $G$ and the signature $(k, B)=(2, (37\quad 17\quad 19\quad 13))$.\\
Thus
\begin{align*}
	s &\equiv k^{G}\Mod{p}\\
	&\equiv 2^{17} \Mod{31}\\
	&\equiv 4 \Mod{31}.
\end{align*}
\begin{align*}
	K^{-1}=Q_{B_4}^{-2}=
		\begin{pmatrix}
		35 & 6 & 0 & 0\\
		-6 & -1 & 0 & 0\\
		7 & 1 & 1 & 0\\
		-8 & -1 & -2 & 1
	\end{pmatrix}
\end{align*}
Now decryption is $A\leftarrow(C-B)K^{-1} \Mod{p}$
\begin{align*}
	A_1 &\equiv (C_1-B)K^{-1} \Mod{31}\\
	&\equiv\left(\begin{pmatrix}
		19 & 11 & 3 & 15
	\end{pmatrix}-
\begin{pmatrix}
		37 & 17 & 19 & 13
\end{pmatrix}
	\begin{pmatrix}
		35 & 6 & 0 & 0\\
		-6 & -1 & 0 & 0\\
		7 & 1 & 1 & 0\\
		-8 & -1 & -2 & 1
	\end{pmatrix}
\right)\Mod{31}\\
	&\equiv\begin{pmatrix}
		22 & 4 & 11 & 2
	\end{pmatrix}\sim
	\begin{pmatrix}
		W & E & L & C
	\end{pmatrix}
\end{align*}
\begin{align*}
	A_2 &\equiv (C_2-B)K^{-1} \Mod{31}\\
	&\equiv\left(\begin{pmatrix}
		6 & 9 & 23 & 13
	\end{pmatrix}-
	\begin{pmatrix}
	35 & 6 & 0 & 0\\
	-6 & -1 & 0 & 0\\
	7 & 1 & 1 & 0\\
	-8 & -1 & -2 & 1
\end{pmatrix}
\right)\Mod{31}\\
	&\equiv\begin{pmatrix}
		14 & 12 & 4 & 0
	\end{pmatrix}\sim
	\begin{pmatrix}
		O& M& E & A
	\end{pmatrix}
\end{align*}
\begin{align*}
	A_3 &\equiv (C_3-B)\Mod{31}\\
	&\equiv\left(\begin{pmatrix}
		17 & 20 & 4 & 17
	\end{pmatrix}-
\begin{pmatrix}
	37 & 17 & 19 & 13
\end{pmatrix}
		\begin{pmatrix}
		35 & 6 & 0 & 0\\
		-6 & -1 & 0 & 0\\
		7 & 1 & 1 & 0\\
		-8 & -1 & -2 & 1
	\end{pmatrix}
	\right)\Mod{31}\\
	&\equiv\begin{pmatrix}
		13 & 13 & 8 & 4
	\end{pmatrix}\sim
	\begin{pmatrix}
		N& N& I & E
	\end{pmatrix}
\end{align*}
Hence, Bob decipher the plaintext \textbf{WELCOMEANNIE}.
\end{exmp}
\subsection{Strength and security analysis}
 We are working in the special linear group $SL_s(F_p)$ of degree $s$ over a field $F_p$, which consists of all invertible matrices of order $s\times s$ over $F_p$. Thus the order of $SL_s(F_p)$ is 
$$
1/(p-1)(p^s-1)(p^s-p)(p^s-p^2)\dots(p^s-p^{s-1}).
$$
It is more difficult to break the scheme due to the order of the key matrix and its large power $k$.
\begin{exmp}
	Consider the key matrix $K=Q^{17}_{B_{15}}$ over a field $F_{31}$. That is, $s=15$ and $p=31$. Thus the cardinality of invertible matrices with determinant $1$ of order $15\times 15$ over $F_{31}$ is
	\begin{align*}
		|SL_{15}(F_{31})|&=1/(31-1)(31^{15}-1)(31^{15}-31)(31^{15}-31^2)\dots(31^{15}-31^{14})\\
		&=1.160251664216324177237764\times 10^{334}.
	\end{align*}
In this case, we need to check $10^{334}$ invertible matrices with determinant $1$. It is impossible, if we deal with huge order $s$.
\end{exmp}
\begin{rem}
	The time complexity of matrix multiplication in worst case is $\mathcal{O}(n^3)$, where $\mathcal{O}$ represents big $\mathcal{O}$ notation \cite{koblitz}. But in the case of generalized balancing matrices the time complexity reduces to $\mathcal{O}(n)$.
\end{rem}
\section{Conclusions}
In this paper, we consider balancing sequences and prove that under the assumption of the $abc$ conjecture for the number field $\mathbb{Q}(\sqrt{2})$, there are at least O(log x) as many balancing non-Wieferich primes $p$ such that $p\equiv\pm1\Mod{r}$ for any fixed integer $r\geq2$.

Further, we suggest a public key cryptosystem using Affine-Hill cipher and generalized balancing matrix $Q_{B_s}^{k}$ with a large power $k$. We propose a key formation (i.e., exchange of the key matrix $K=Q_{B_s}^{k}$ of order $s\times s$ for the encryption-decryption scheme with the help of balancing sequences under prime modulo). In this scheme, instead of exchanging key matrix, we simply exchange a pair of numbers $(s,k)$, which results in a wide-key space and lower time and space complexity.

\medskip
\noindent
\textbf{Acknowledgement.}
	The author I. Mumtaj Fathima would like to express her gratitude to Maulana Azad National Fellowship for minority students, UGC. This research work is supported by MANF-2015-17-TAM-56982, University Grants Commission (UGC), Government of India.
	
	%\pagebreak
	
\end{document}